\documentclass[12pt,english,reqno]{amsart}  

\usepackage[T1]{fontenc} 
\usepackage{geometry} 
\usepackage{amssymb,amsmath,amsthm} 

\usepackage{pdfpages}
\usepackage{graphicx}
\usepackage{tikz} 
\usetikzlibrary{arrows,automata} 

\usepackage{hyperref}
\hypersetup{
    colorlinks=true,
    linkcolor=blue,
    filecolor=magenta,      
    urlcolor=cyan,
}

\usepackage[utf8]{inputenc}

\newtheorem{theorem}{Theorem}[section]

\newtheorem{lemma}[theorem]{Lemma}

\newtheorem{remark}[theorem]{Remark}

\newcommand{\bfS}[1]{\mathbf{S}^{#1}} 
\newcommand{\bfR}[1]{\mathbf{R}^{#1}} 

\newcommand{\R}{\mathbb R}

\title{Navigating the Space of Compact CMC Hypersurfaces in Spheres, Part II}
\author{Oscar Perdomo}
\address{Central Connecticut State University}
\email{perdomoosm@ccsu.edu}
\date{\today}

\begin{document}

\maketitle

\begin{abstract}
In \(\bfR{3}\), let \(
  \mathcal{M}
  =\bigl\{\lvert x-(2m,0,0)\rvert=1 : m\in\mathbb{Z}\bigr\}
\)
be the infinite union of unit spheres whose centers lie at even integers on the \(x\)-axis; every pair of consecutive spheres touches at \((2m+1,0,0)\).  Desingularizing these point contacts yields Delaunay’s classical constant-mean-curvature surfaces (unduloids and nodoids).

Motivated by this picture, we construct an analogue in the unit sphere \(\bfS{4}\).  We begin with the piecewise-smooth hypersurface \(M\subset\bfS{4}\) obtained by gluing two carefully chosen totally umbilical 3-spheres to two specific Clifford hypersurfaces, all four components sharing the same constant mean curvature and meeting along four disjoint circles.

We provide numerical evidence that these circles can be desingularized: there exists a smooth one-parameter family 
\(\{\Sigma_{b}\}_{b\in(0,B)}\subset\bfS{4}\)
of constant-mean-curvature hypersurfaces such that \(\Sigma_{b}\to M\) as \(b\to0\).  The mean curvature \(H(b)\) varies smoothly along the family and vanishes at a single non-embedded minimal member.  Moreover, there is a threshold \(B_{1}\in(0,B)\) such that when \(b<B_{1}\) the hypersurface \(\Sigma_{b}\) is embedded (``unduloid type''), whereas for \(b\ge B_{1}\) it is non-embedded (``nodoid type'').  As \(b\nearrow B\), the hypersurfaces converge to a minimal hypersurface with two singular points.
\end{abstract}

`
\section{Introduction}

Constant-mean-curvature (CMC) hypersurfaces of revolution in \(\bfS{n+1}\) are well understood; see, for instance, \cite{PRe}.   By contrast, 
the landscape of CMC hypersurfaces 
arising from the particular case of the more general  \emph{generalized rotational immersions}
\cite{P},
\begin{equation}\label{gh}
\varphi: \bfS{\ell}\times\bfS{m}\times\bfS{1} \longrightarrow \bfS{\ell+m+2},
\quad
\varphi(y,z,t)=(f_{2}(t)y,\,f_{3}(t)z,\,f_{1}(t)),
\end{equation}
is far from fully charted.  In \cite{CS}, Carlotto and Schulz treat the symmetric case \(m=\ell\) and produce infinitely many minimal (\(H=0\)) examples, exactly one of which is embedded for each \(\ell\).  The work in \cite{HW} extends that embedded minimal example in one direction to a \emph{one-parameter family} of CMC hypersurfaces with positive mean curvature (\(H>0\)), and \cite{LW} constructs complementary examples with \(H<0\), roughly half of which continue the minimal examples of \cite{CS} in the opposite direction considered in \cite{HW}.  A numerical study of the full deformation in \emph{both} directions—covering not only the symmetric case but also the asymmetric regime \(m\neq\ell\)—appears in \cite{Pa}.  Finally, \cite{P} revisits the embedded minimal examples of \cite{CS}, extends them to \(m\neq\ell\), and gives an explicit procedure for computing the spectra of both the Laplacian and the stability (Jacobi) operator.

To the author’s surprise, the embedded CMC hypersurfaces of the form \eqref{gh} described in \cite{Pa,HW,LW} are \emph{not} exhaustive.  In this paper, we present numerical evidence of a new one-parameter family containing both embedded and non-embedded examples.  Just as the classical Delaunay surfaces in \(\bfR{3}\) can be viewed as emerging from a chain of mutually tangent unit spheres, our family in \(\bfS{4}\) emerges from a carefully chosen piecewise-CMC hypersurface composed of tangent Clifford hypersurfaces and totally umbilical 3-spheres, all sharing the same mean curvature.  No analogous Euclidean construction is possible, since in \(\bfR{3}\) tangent spheres and cylinders do \emph{not} have equal mean curvature.

Section~\ref{sec:cmcsing} describes the two singular limit hypersurfaces of our family.  The first limit is a piecewise-CMC hypersurface in \(\bfS{4}\) whose singular set is a union of four circles (a co-dimension-two submanifold); this parallels the Delaunay construction in \(\bfR{3}\), where the singular set is a discrete set of points.  The second limit is a minimal hypersurface in \(\bfS{4}\) with singularities at exactly two points.

Section~\ref{sec:ode} presents the governing differential equation and explains the numerical continuation method used to generate our family, including how the singular limit hypersurfaces correspond to solutions with singularities of the ODE.  Finally, Section~\ref{sec:S4} details the resulting one-parameter family in \(\bfS{4}\) and analyzes its geometric properties.

\section{The singular CMC generating hypersurfaces}\label{sec:cmcsing}

This section constructs the two singular limit hypersurfaces of our one-parameter family of CMC hypersurfaces in \(\bfS{4}\).

\subsection{Piecewise–CMC (totally umbilical/Clifford) limit hypersurface}
In this subsection, we describe the first limit configuration, which is more intricate.  We construct in detail the piecewise-CMC hypersurface from which our family of CMC immersions emanates.  Although its components are Clifford and totally umbilical hypersurfaces, we also include the generalized-rotational examples to illustrate how they fit together.  We consider the following four immersions into the unit sphere \(\bfS{4}\subset\bfR{5}\):

\begin{align*}
M_{1} &= \bigl\{\,x=(x_{1},\dots,x_{5})\in\bfR{5} : x_{1}^{2}+\cdots+x_{5}^{2}=1,\;x_{5}=c\,\bigr\},\\
      &\quad c = \pm\sqrt{1-r_{1}^{2}},\quad r_{1}\in(0,1),\\[6pt]
M_{2} &= \bigl\{\,(r_{2}\cos\theta_{1},\,r_{2}\sin\theta_{1},\,\sqrt{1-r_{2}^{2}}\;u)\;:\;\theta_{1}\in\R,\;u\in\bfS{2}\bigr\},\\
      &\quad r_{2}\in(0,1)\text{ constant},\\[6pt]
M_{3} &= \bigl\{\;(\sqrt{1-r_{3}^{2}}\,u_{1},\,\sqrt{1-r_{3}^{2}}\,u_{2},\,r_{3}\cos\theta_{2},\,r_{3}\sin\theta_{2},\,\sqrt{1-r_{3}^{2}}\,u_{3}):\\
      &\quad \theta_{2}\in\R,\;(u_{1},u_{2},u_{3})\in\bfS{2}\bigr\},\\
      &\quad r_{3}\in(0,1)\text{ constant},\\[6pt]
M_{4} &= \bigl\{\,(f_{2}(t)\cos\theta_{1},\,f_{2}(t)\sin\theta_{1},\,f_{3}(t)\cos\theta_{2},\,f_{3}(t)\sin\theta_{2},\,f_{1}(t)):\\
      &\quad \theta_{1},\theta_{2}\in\R,\;f_{1}(t)^{2}+f_{2}(t)^{2}+f_{3}(t)^{2}=1\bigr\}.
\end{align*}

Here, \(M_{1}\) is a totally umbilical 3–sphere in \(\bfS{4}\) with mean curvature
\[
H_{1} \;=\; \frac{\sqrt{1 - r_{1}^{2}}}{r_{1}}.
\]
For \(j=2,3\), each \(M_{j}\) is isometric to the Clifford hypersurface
\[
\bfS{2}\bigl(\sqrt{1 - r_{j}^{2}}\bigr)\times\bfS{1}(r_{j}),
\]
with mean curvature
\[
H_{j} \;=\; \frac{1}{3}\Bigl(\frac{2r_{j}}{\sqrt{1 - r_{j}^{2}}}
  \;-\;\frac{\sqrt{1 - r_{j}^{2}}}{r_{j}}\Bigr).
\]

When the “profile curve”
\[
\alpha(t) \;=\;(f_{3}(t),\,f_{2}(t),\,f_{1}(t))
\]
is a smooth periodic curve in \(\bfS{2}\) with \(f_{2}(t)>0\) and \(f_{3}(t)>0\) for all \(t\), the immersion \(M_{4}\) defines
\(\bfS{1}\times\bfS{1}\times\bfS{1}\to\bfS{4}\).  At special choices one may drop periodicity and allow \(f_{2}(t)\) or \(f_{3}(t)\) to vanish at isolated values of \(t\) while still obtaining an immersion.  For example:
\begin{itemize}
  \item If \(f_{1}(t)=\pm\sqrt{1 - r_{1}^{2}}\) is constant and \((f_{2}(t),f_{3}(t))\) trace the circle of radius \(r_{1}\), then \(M_{4}=M_{1}\).
  \item If \(f_{2}(t)=r_{2}\) and \((f_{1}(t),f_{3}(t))\) trace the circle of radius \(\sqrt{1 - r_{2}^{2}}\), then \(M_{4}=M_{2}\).
  \item If \(f_{3}(t)=r_{3}\), then \(M_{4}=M_{3}\).
\end{itemize}

Consider the intersection \(M_{1}\cap M_{2}\).  There are three cases:

\begin{enumerate}
  \item \(r_{2}>r_{1}\).  Then 
  \[
    \sqrt{1-r_{2}^{2}}<\sqrt{1-r_{1}^{2}},
    \quad\text{so}\quad
    M_{1}\cap M_{2}=\emptyset.
  \]

  \item \(r_{2}=r_{1}\).  A point of \(M_{2}\) lies in \(M_{1}\) exactly when \(u_{3}=1\) (i.e.\ \(u=(0,0,1)\)), and
  \[
    M_{1}\cap M_{2}
    =\{(r_{2}\cos\theta_{1},\,r_{2}\sin\theta_{1},\,0,\,0,\,\sqrt{1-r_{2}^{2}})\},
  \]
  which is a circle of radius \(r_{2}\).

  \item \(r_{2}<r_{1}\).  Define
  \[
    r_{12}
    =\frac{\sqrt{r_{1}^{2}-r_{2}^{2}}}{\sqrt{1-r_{2}^{2}}}.
  \]
  Then
  \[
    M_{1}\cap M_{2}
    =\{(r_{2}\cos\theta_{1},\,r_{2}\sin\theta_{1},\,r_{12}\cos\theta_{2},\,r_{12}\sin\theta_{2},\,\sqrt{1-r_{2}^{2}})\},
  \]
  which is a torus isometric to
  \[
    \bfS{1}(r_{2})\times\bfS{1}(r_{12}).
  \]
\end{enumerate}

\subsubsection{Determining the radii for the singular CMC generating hypersurfaces}

\[
  \frac{\sqrt{1 - r^{2}}}{r}
  =
  \frac{1}{3}\Bigl(\frac{2r}{\sqrt{1 - r^{2}}}
    - \frac{\sqrt{1 - r^{2}}}{r}\Bigr),
  \quad r>0,
\]
whose unique solution is
\[
  r^{2} = \frac{2}{3},
  \quad
  r = \sqrt{\frac{2}{3}}.
\]

With this choice of \(r\), define
\begin{equation}\label{mM}
  M \;=\; M_{1}^{+}\;\cup\; M_{1}^{-}\;\cup\; M_{2}\;\cup\; M_{3}
      \;\setminus\; (\gamma_{1}\cup\gamma_{2}\cup\gamma_{3}\cup\gamma_{4}),
\end{equation}
where
\begin{align*}
  M_{1}^{\pm} &= \bigl\{\,x\in\bfS{4} : x_{5}=\pm\sqrt{1 - r^{2}}\,\bigr\},\\
  M_{2}       &= \bigl\{\,x\in\bfS{4} : x_{1}^{2}+x_{2}^{2}=r^{2}\,\bigr\},\\
  M_{3}       &= \bigl\{\,x\in\bfS{4} : x_{3}^{2}+x_{4}^{2}=r^{2}\,\bigr\}.
\end{align*}

The intersection circles are
\[
\begin{aligned}
\gamma_{1}&=\{(r\cos\theta,\,r\sin\theta,\,0,\,0,\,\sqrt{1-r^{2}})\colon\theta\in\mathbb{R}\},\\
\gamma_{2}&=\{(0,\,0,\,r\cos\theta,\,r\sin\theta,\,\sqrt{1-r^{2}})\colon\theta\in\mathbb{R}\},\\
\gamma_{3}&=\{(r\cos\theta,\,r\sin\theta,\,0,\,0,\,-\sqrt{1-r^{2}})\colon\theta\in\mathbb{R}\},\\
\gamma_{4}&=\{(0,\,0,\,r\cos\theta,\,r\sin\theta,\,-\sqrt{1-r^{2}})\colon\theta\in\mathbb{R}\}.
\end{aligned}
\]


Each component \(M_{1}^{\pm}\), \(M_{2}\), and \(M_{3}\) has the same constant mean curvature
\[
  \frac{\sqrt{1 - r^{2}}}{r}
  = \frac{1}{3}\Bigl(\frac{2r}{\sqrt{1 - r^{2}}}
                     - \frac{\sqrt{1 - r^{2}}}{r}\Bigr)
  = \frac{1}{\sqrt{2}}.
\]
Hence \(M\) is a piecewise‐smooth CMC hypersurface.  Since \(2r^{2}=4/3>1\), the Clifford pieces satisfy 
\[
  M_{2}\cap M_{3}=\emptyset,
\]
and the remaining intersections are exactly the four circles:
\[
  M_{2}\cap M_{1}^{+}=\gamma_{1},\quad
  M_{3}\cap M_{1}^{+}=\gamma_{2},\quad
  M_{2}\cap M_{1}^{-}=\gamma_{3},\quad
  M_{3}\cap M_{1}^{-}=\gamma_{4}.
\]
Thus \(M\) is our singular hypersurface in \(\bfS{4}\).  Section~\ref{sec:S4} will show how the new family converges to \(M\).

\subsection{The second limit hypersurface}

In this subsection, we describe the second singular limit in the general \(\ell\)-parameter setting (for \(\bfS{4}\), \(\ell=m=1\)).  Numerical evidence shows that our one-parameter CMC family, which emanates from the piecewise-CMC hypersurface \(M\) of \eqref{mM}, terminates at the singular minimal hypersurface

\begin{eqnarray}\label{mMf}
M_{f}
=\bigl\{\,x\in\bfS{2\ell+1}:x_{1}^{2}+\cdots+x_{\ell+1}^{2}
 = x_{\ell+2}^{2}+\cdots+x_{2\ell+2}^{2}\bigr\},
\end{eqnarray}

which is smooth away from the two poles \(\pm(0,\dots,0,1)\) and may be viewed as a cylinder over the minimal Clifford hypersurface.  Moreover, the parametrization
\[
\phi(u,v,t)
=\Bigl(\tfrac{\cos t}{\sqrt{2}}\,\xi(u),\;\tfrac{\cos t}{\sqrt{2}}\,\xi(v),\;\sin t\Bigr),
\]
with \(\xi\colon W\subset\bfR{\ell}\to\bfS{\ell}\) a local chart on \(\bfS{\ell}\), exhibits \(M_{f}\) as a generalized-rotational hypersurface with two isolated singularities and permits explicit computation of its principal curvatures.

\section{The ODE}\label{sec:ode}

As expected, the constant–mean–curvature condition reduces to an ordinary differential equation.  As in \cite{P,Pa}, we study the curve
\[
\beta(t)=(f_{1}(t),f_{2}(t))\in D_{+}\subset\bfR{2},
\quad
D_{+}=\{(x,y):x^{2}+y^{2}<1,\;y>0\},
\]
rather than the full curve \((f_{1}(t),f_{2}(t),f_{3}(t))\in\bfS{2}\).  We assume \(\beta\) is parameterized by arc length, so there is an angle function \(\theta(t)\) with
\[
f_{1}'(t)=\cos\theta(t),
\quad
f_{2}'(t)=\sin\theta(t).
\]
Set
\[
g = f_{2}\cos\theta - f_{1}\sin\theta,\quad
h = \sqrt{1 - g^{2}},\quad
f_{3} = \sqrt{1 - f_{1}^{2} - f_{2}^{2}}.
\]
Then (see \cite{Pa}) the immersion \(\varphi\) of \eqref{gh} has constant mean curvature \(H\) if and only if
\[
\theta'
= \frac{h^{2}}{f_{3}^{2}f_{2}}
  \Bigl(\ell\cos\theta - n\,f_{2}g + nH\,f_{2}h\Bigr),
\quad
n=\ell+m+1.
\]

\begin{theorem}\label{thm:ode}
Under the hypotheses of \(\varphi\) in \eqref{gh}, the functions \((f_{1},f_{2},\theta)\) satisfy the autonomous system
\[
\begin{cases}
f_{1}' = \cos\theta,\\[4pt]
f_{2}' = \sin\theta,\\[4pt]
\theta'
= \displaystyle\frac{h^{2}}{f_{3}^{2}f_{2}}
  \bigl(\ell\cos\theta - n\,f_{2}g + nH\,f_{2}h\bigr),
\end{cases}
\]
where \(g=f_{2}\cos\theta - f_{1}\sin\theta\) and \(h=\sqrt{1-g^{2}}\).  This system is equivalent to \(\varphi\) having constant mean curvature \(H\).
\end{theorem}

Let us continue finding out some explicit solutions of the system

\begin{lemma} \label{lemma1} If $f_2(t)=r_2$ is a solution of the ODE in Theorem \ref{thm:ode} then, $|H|=\frac{|nr_2^2-\ell|}{nr_1r_2}$ with $r_1=\sqrt{1-r_2^2}$.

\end{lemma}
\begin{proof}
Since $(f_1(t),f_2(t))$ is parametrized by arc length and $\sin(\theta)=0$ then, either $\cos(\theta)=1$ and $f_1(t)=t$ or $\cos(\theta)=-1$ and $f_1(t)=-t$. We can disregard the constant because the system is autonomous. In this case 

$$g=f_2\cos\theta -f_1\sin\theta=r_2\cos\theta\,\hbox{ and }\, h=r_1=\sqrt{1-r_2^2}$$

Since $\theta$ is constant, then $\ell \cos\theta-nf_2g+nHf_2h=0$, therefore

$$\ell \cos\theta-nr_2^2\cos\theta+nHr_2r_1=0$$ and then 

$$|H|=\frac{|nr_2^2-\ell|}{nr_1r_2}$$

\end{proof}

\begin{lemma}\label{lemma2} If  $f_1(t)=c$ with $|c|=r_1$ is a solution of the ODE in Theorem \ref{thm:ode} then, $|H|=\frac{r_1}{r_2}$ with $r_2=\sqrt{1-r_1^2}$
\end{lemma}
\begin{proof}
Since $(f_1(t),f_2(t))$ is parametrized by arc length and $\cos(\theta)=0$ then, either $\sin(\theta)=1$ and $f_2(t)=t$ or $\sin(\theta)=-1$ and $f_2(t)=-t$. In this case 

$$g=f_2\cos\theta -f_1\sin\theta=-r_1\sin\theta\,\hbox{ and }\, h=r_2=\sqrt{1-r_1^2}$$

Since $\theta$ is constant, then $\ell \cos\theta-nf_2g+nHf_2h=0$, therefore

$$-g+Hh=r_1\sin \theta+Hr_2=0$$ 

and then 

$$|H|=\frac{r_1}{r_2}$$

\end{proof}
\begin{lemma}\label{lemma3} If $(f_1(t),f_2(t))$ is a solution of the ODE in Theorem \ref{thm:ode} that satisfies $f_3(t)=\sqrt{1-f_1(t)^2-f_2(t)^2}=r_2$ then, $|H|=\frac{|\ell+1-nr_1^2|}{nr_1r_2}$ with $r_1=\sqrt{1-r_2^2}$
\end{lemma}
\begin{proof}
Since $(f_1(t),f_2(t))$ is parametrized by arc length and $|(f_1(t),f_2(t))|=r_1$  then, either $f_1=r_1\sin(\frac{t}{r_1})$ and $f_2=r_1\cos(\frac{t}{r_1})$ or
$f_1=r_1\sin(-\frac{t}{r_1})$ and $f_2=r_1\cos(-\frac{t}{r_1})$. If  

$$f_1=r_1\sin(\frac{t}{r_1})\, \hbox{ and }\, f_2=r_1\cos(\frac{t}{r_1})$$

then $\theta(t)=-\frac{t}{r_1}$,

$$g=f_2\cos\theta -f_1\sin\theta=r_1\cos(\frac{t}{r_1})\cos(-\frac{t}{r_1})-r_1\sin(\frac{t}{r_1})\sin(-\frac{t}{r_1})=r_1$$

 $h=r_2$ and 

$$\ell \cos\theta-nf_2g+nHf_2h=\ell \cos(-\frac{t}{r_1})-nr_1^2\cos(\frac{t}{r_1})+nHr_2r_1\cos(\frac{t}{r_1})$$

Therefore, the equation $\theta'=\frac{h^2}{f_3^2f_2}\left( \ell \cos\theta-nf_2g+nHf_2h\right)$ reduces to

$$-\frac{1}{r_1}=\frac{r_2^2}{r_2^2r_1}\left(\ell-nr_1^2+nHr_2r_1 \right)$$

and then

$$|H|=\frac{|\ell+1-nr_1^2|}{nr_1r_2}$$

We can check that the case $f_1=r_1\sin(-\frac{t}{r_1})$ and $f_2=r_1\cos(-\frac{t}{r_1})$ leads to the same conclusion.
\end{proof}

The next theorem shows we could have gotten the values for $H$ and the radii of the totally umbilical spheres and Clifford hypersurfaces in the piecewise-CMC hypersurface $M$  defined in \eqref{mM} form the differential equation above.

\begin{theorem}\label{thm:pwcmc}

If $\ell=m$ and 

\begin{eqnarray}\label{eqn:vHandrs}
H=-\sqrt{\frac{n+1}{3n-1}},\quad r_1=\sqrt{\frac{n+1}{4 n}} \quad\hbox{and}\quad r_2=\sqrt{\frac{3n-1}{4 n}}
\end{eqnarray}

there exists a solution of the ODE in Theorem \ref{thm:ode} that satisfies the following four conditions

\begin{enumerate}
\item
$f_2(t)=r_2$ is a solution of the CMC equation.
\item
There exists a solution of the CMC equation that makes 
$$f_3(t)=\sqrt{1-f_1^2(t)-f_2(t)^2}=r_2$$
\item
$f_1(t)=r_1$ is another solution of the CMC equation.
\item
$f_1(t)=-r_1$ is another solution of the CMC equation.
\end{enumerate}
Moreover, if $\ell\ne m$, then we cannot find numbers $H$, $r_1$ and $r_2$ that satisfies the conditions above.
\end{theorem}

\begin{proof} Let us start assuming that there exists a solution that satisfies the four properties above. From Lemma \ref{lemma1} and \ref{lemma3}  we have that 

$$\frac{|\ell+1-nr_1^2|}{nr_1r_2}=\frac{r_1}{r_2}.$$

From the equation above we get that $\ell+1-nr_1^2<0$ then $\ell+1=0$, therefore and $\ell+1-nr_1^2>0$ and we $\frac{|\ell+1-nr_1^2|}{nr_1r_2}=\frac{r_1}{r_2}$ 
reduces to

\begin{eqnarray}\label{landn}
1+\ell=2 nr_1^2
\end{eqnarray}

Now, using Lemma \ref{lemma2} we have that 

$$\frac{|\ell-nr_2^2|}{nr_1r_2}=\frac{r_1}{r_2}.$$

This time, notice that if $\ell-nr_2^2>0$, then $\ell=nr_1^2+nr_2^2=n$ which is not possible because $\ell<n$. Therefore $\ell-nr_2^2<0$ and

$$nr_2^2=\ell+nr_1^2=3nr_1^2-1=3n-3nr_2^2-1$$

and 

$$r_2=\sqrt{\frac{3n-1}{4n}}\, \hbox{ and } \, r_1=\sqrt{\frac{n+1}{4n}}$$

Finally from the equation 

$$|H|=\frac{\ell+1-nr_1^2}{nr_1r_2}=\frac{nr_2^2-\ell}{nr_1r_2}$$

We conclude that 

$$|H|=\sqrt{\frac{n+1}{3 n-1}}\quad \hbox{ and }\quad n=2\ell+1.$$

Finally if $H$, $r_1$ and $r_2$ are as in Equation \eqref{eqn:vHandrs}, we can check that the solutions

\begin{enumerate}
\item
$f_2(t)=r_2$ and $f_1=-t$ 
\item

$f_1(t)=r_1\sin(\frac{t}{r_1})$ and $f_2(t)=r_1\cos(\frac{t}{r_1})$
\item
$f_1(t)=r_1$ and $f_2(t)=t$ 
\item
$f_1(t)=-r_1$ and $f_2(t)=-t$

\end{enumerate}

satisfy the four conditions above.

\end{proof}

\begin{remark}\label{rem:singsol}
The previous Theorem shows that there exists a periodic solution of  the ODE in Theorem \ref{thm:ode} defined on the interval $[0,2T]$ with singularities at $T_1,T_2,T_3$ and $T_4$  with

\begin{eqnarray*}
T_1&=&\frac{1}{4} \pi  \sqrt{\frac{1}{n}+1}, \\
T_2&=&\frac{1}{4} \pi  \sqrt{\frac{1}{n}+1}+\frac{1}{2} \sqrt{3-\frac{1}{n}}\\
T_3&=&\frac{1}{4} \pi  \sqrt{\frac{1}{n}+1}+\sqrt{\frac{1}{n}+1}+\frac{1}{2} \sqrt{3-\frac{1}{n}},\\
T_4&=&\frac{1}{4} \pi  \sqrt{\frac{1}{n}+1}+\sqrt{\frac{1}{n}+1}+\sqrt{3-\frac{1}{n}}\\
2 T&=&\frac{1}{2} \pi  \sqrt{\frac{1}{n}+1}+\sqrt{\frac{1}{n}+1}+\sqrt{3-\frac{1}{n}}
\end{eqnarray*}

For this solution, the function $f_1(t)$ and $f_2(t)$ are continuous with $f_1(0)=0$ and $f_2(0)=r_1$ the function $\theta(t)$ is piecewise linear but it just jump discontinuities at each singularity. Graph 

\end{remark}

\begin{figure}[h]
\centerline{\includegraphics[scale=0.52]{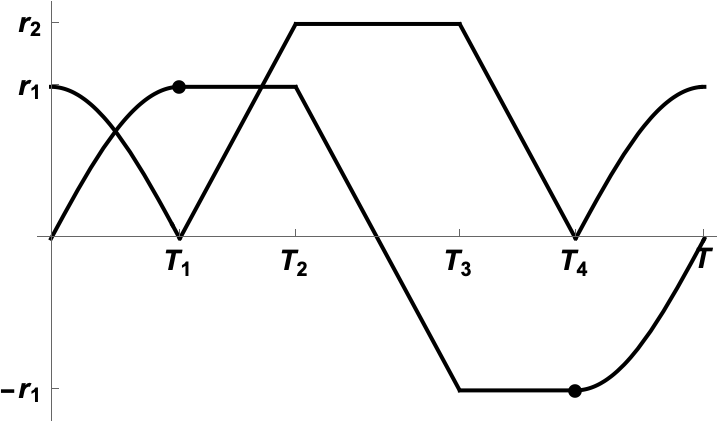} \hskip.2cm \includegraphics[scale=0.52]{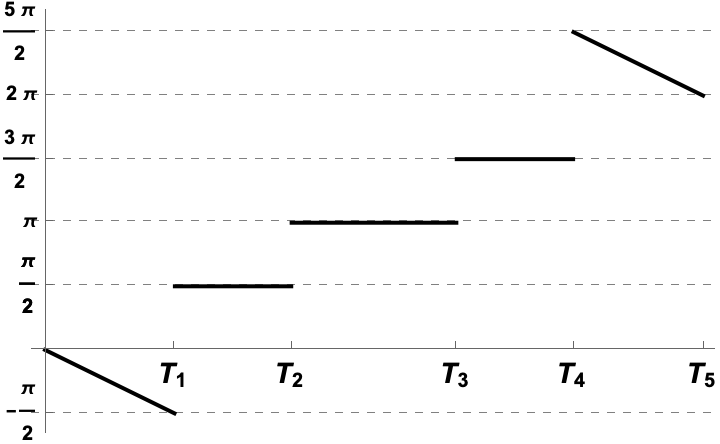}}
\caption{The piecewise–CMC hypersurface defined in Equation~\ref{mM} corresponds to a periodic solution \((\theta(t),f_{1}(t),f_{2}(t))\) that exhibits four singularities.  The graphs of \(f_{1}\) and \(f_{2}\) are shown on the left, and the graph of \(\theta\) is shown on the right.}
\label{fig:graphf1andf2}
\end{figure}

\begin{remark}\label{rem:singsolellipse}
The length of the Ellipse 

$$E=\{(x,y)\in \bfR{2}:x^2+\frac{1}{2}y^2=1\}$$

 can be given explicitly in terms of elliptic functions. We have that 
 
 $$2T=2\sqrt{2} \hbox{EllipticE}(-1)=2\sqrt{2}\int_0^{\frac{\pi}{2}}\sqrt{\sin ^2(t)+1}\, dt\approx 5.402575524$$

Let us denote by 

\begin{eqnarray*}
& & T_1=\frac{T}{4}\approx  1.350643881,\quad T_2=2 T1=T,\quad T_3=3T_1\quad \hbox{and}\\\
\quad \quad  \quad \quad && \quad T_4=4 T_1=2T \hbox{ the length of the Ellipse $E$.}
\end{eqnarray*}

Since the manifold $M_f$ given in Equation \eqref{mMf} is minimal then we have that if $(g_1(t),g_2(t))$ is a  clockwise parametrization of the ellipse $E$ by arc length starting at
the point $(0,\frac{1}{\sqrt{2}})$ from $0$ to $T_1$ and $(g_3(t),g_4(t))$ is a  counter clockwise parametrization of the ellipse $E$ by arc length starting at
the point  $(1,0)$ from $T_1$ to $3T_1=T_3$ and $(g_5(t),g_6(t))$ is a  clockwise parametrization of the ellipse $E$ by arc length starting at
the point  $(-1,0)$ from $T_3$ to $4T_1=2T$, then the functions

\[
f_{1}(t)=
\begin{cases}
  g_{1}(t), & t\in[0,T_{1}],\\[6pt]
  g_{3}(t), & t\in[T_{1},T_{3}]\\[6pt]
  g_{5}(t), & t\in[T_{3},T_{5}].
\end{cases}
\]

and

\[
f_{2}(t)=
\begin{cases}
  g_{2}(t), & t\in[0,T_{1}],\\[6pt]
  g_{4}(t), & t\in[T_{1},T_{3}]\\[6pt]
  g_{5}(t), & t\in[T_{3},T_{5}].
\end{cases}
\]
a periodic solution of  the ODE in Theorem \ref{thm:ode} defined on the interval $[0,2T]$ with singularities at $T_1,T_2,T_3$ and $T_4$. Figure \ref{fig:gfMf}

\begin{figure}[h]
\centerline{\includegraphics[scale=0.72]{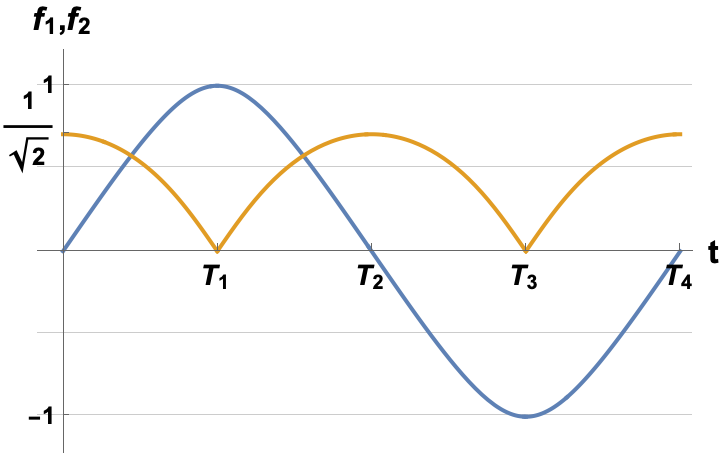}}
\caption{The  CMC hypersurface $M_f$ defined in Equaation  \ref{fig:gfMf} can be regarded as a periodic  solution of the $(\theta(t),f_1(t), f_2(t)$ with three singularities. Here we show  the graphs of $f_1$ and $f_2$ associated with this solution with singularities }
\label{fig:gfMf}
\end{figure}
\end{remark}

\section{The family in $\bfS{4}$}\label{sec:S4}

When we consider solution of the ODE in Theorem \ref{thm:ode} that satisfies that
\begin{eqnarray}\label{ic}
f_1(0)=0,\quad 0<f_2(0)=a<1, \quad \theta^\prime(0)=0
\end{eqnarray} 

we have shown  in \cite{P} and \cite{Pa} that, if for some positive $T$ we have a solution that satisfies $f_1(T)=0$ and $\theta(T)=\pi$, then this solution can be extended to a $2T$-periodic solution. Notice that when we have found a value $T$ then, the following vector will determine this solution

$$Z=(a,H,T)\, .$$

With this in mind, we need to search for points in the three dimensional space 

\begin{eqnarray}\label{aHTspace}
\Omega=\{(a,H,T)\in\bfR{3}:0<a<1\}\, ,
\end{eqnarray}

that are solution of the   ODE in Theorem \ref{thm:ode} and satisfy the following system of two equation and three variables

\begin{eqnarray}\label{theaHTsytem}
\begin{cases} F_1(a,H,T)&=0\\
\Theta(a,H,T)&=\pi
  \end{cases}
  \end{eqnarray}

where 

$$t\longrightarrow F_1(a,H,t),t\longrightarrow F_2(a,H,t),t \longrightarrow \Theta(a,H,t) $$

 is the solution of the system with $F_1(a,H,0)=0$, $F_2(a,H,0)=a$ and  $\Theta(a,H,0)=0$. 
 
 \begin{remark}\label{rem:defLambda}
 By the implicit function theorem we have that anytime we have a solution of the System \eqref{theaHTsytem}, if the gradient of the the functions $F_1$ and $\Theta$ are linearly independent then the solution can locally extended to a curve of points satisfying the system. In \cite{Pa} the author study one curve that solve the system and in this paper we will study a different curve, that we will call $\Lambda$ that solve the system as well.
 \end{remark}
 
 From  Remark \ref{rem:singsol} we have that the point
 
 $$Z_0=\left(\frac{1}{\sqrt{3}},-\frac{1}{\sqrt{2}},\frac{2 \sqrt{2}+\pi +2}{2 \sqrt{3}}\right)\approx (0.57735, -0.707107, 2.30075)$$
 
 would be a solution of the system \eqref{theaHTsytem} if only the solutions shown in Figure \ref{fig:graphf1andf2} were smooth. Also from Remark \ref{rem:singsolellipse} we have that the point

\begin{eqnarray}\label{pZf}
Z_f=\left(\frac{1}{\sqrt{2}},0,\sqrt{2}\hbox{ EllipticE}(-1)\right)\approx (0.707107, 0, 2.70129)
\end{eqnarray}
  
  would be a solution of the system \eqref{theaHTsytem} if only the solutions shown in Figure \ref{fig:gfMf} were smooth.

 We can show numerically that the points $Z_0$ and $Z_f$ are the limit point of a curve of points in $Z\in \Lambda$ that solves the system of equations  \eqref{theaHTsytem}. Figure \ref{fig:Z3S4} shows the curve $\Lambda$.

 \begin{figure}[h]
\centerline{\includegraphics[scale=0.45]{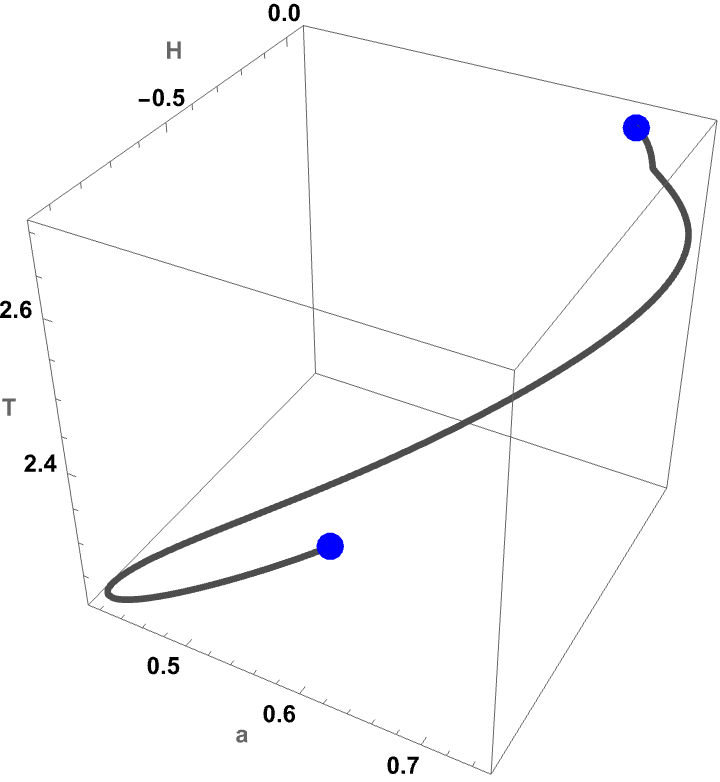}\hskip1cm\includegraphics[scale=0.45]{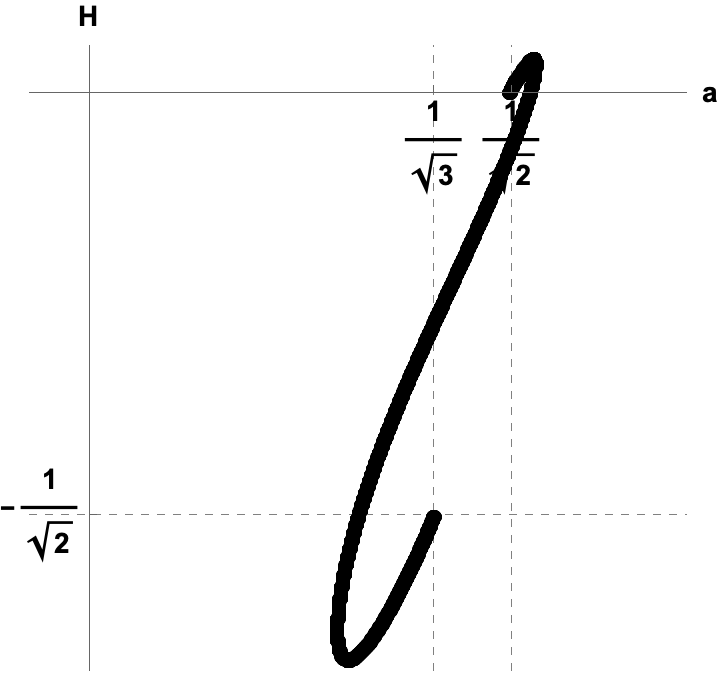}}
\caption{Graph of the curve $\Lambda$ that solve the system of equations  \eqref{theaHTsytem}. Starting form the limit point $Z_0$ (the one with H$H<0$, we notice that the values of $H$ start decreasing reaching a minimum value of $H$, then, they increase until they reach a positive maximum and finally $H$ decreases to zero to approach the limit point $Z_f$. We also show the projection of $\Lambda$ on the $a-H$ plane. } 
\label{fig:Z3S4}
\end{figure}
 
We use the continuation method to numerically get points in $\Lambda$. Once we have a numerical solution $Z_k=(a_k,H_k,T_k)$ of the system  \eqref{theaHTsytem},  we compute the gradient of the two equations in the system at that point -we  compute $\nabla F_1(Z_0)$ and  $\nabla \Theta(Z_0)$- and then, we look for another solution by moving the point $Z_k$ in the direction perpendicular to both gradients. See \cite{P}, \cite{Pa}.

  \begin{remark}
All the solutions considered in this paper satisfy both equations in the system \eqref{theaHTsytem} up to an error of $10^{-7}$. As shown through detailed and lengthy arguments in \cite{Pnm} and \cite{P3bp}, such solutions can, in principle, be proven to be mathematically exact. For the sake of simplicity, we will say that a solution satisfies the system, rather than repeatedly stating that it satisfies the system up to an error of $10^{-7}$.
\end{remark}

 Let us start describing the curve $\Lambda$ that produces the new family of CMC hypersurfaces.
 
 \subsection{Smooth solution very close to $Z_0$} The point
 
 $$Z_1=(a,H,T)=(0.577096, -0.707791, 2.30054)$$ 
 
 solves the system of equations  \eqref{theaHTsytem}. Figure \ref{fig:Z1S4} shows the graphs of the profile curve  and the functions $f_1,f_2$ and $\theta$ associated with the parameters in $Z_1$. Since the profile curve associate with this value of $H$ and this initial condition is embedded then we have an  embedded hypersurface with CMC  $-0.707791$. For the naked eye, if we superpose the graphs of the solutions of $f_1$ and $f_2$ in Figures \ref{fig:graphf1andf2} and \ref{fig:Z1S4}  we could not tell the difference between these to functions, nevertheless the the functions $f_1$ and $f_2$ and $\theta$ in Figures \ref{fig:Z1S4} are $C^\infty$.
 
 \begin{figure}[h]
\centerline{\includegraphics[scale=0.42]{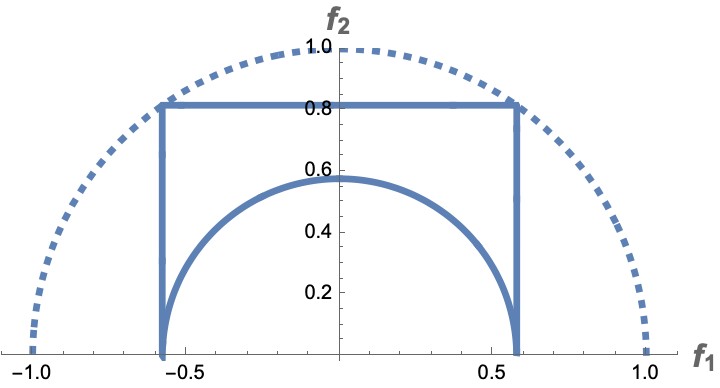} \hskip.2cm \includegraphics[scale=0.42]{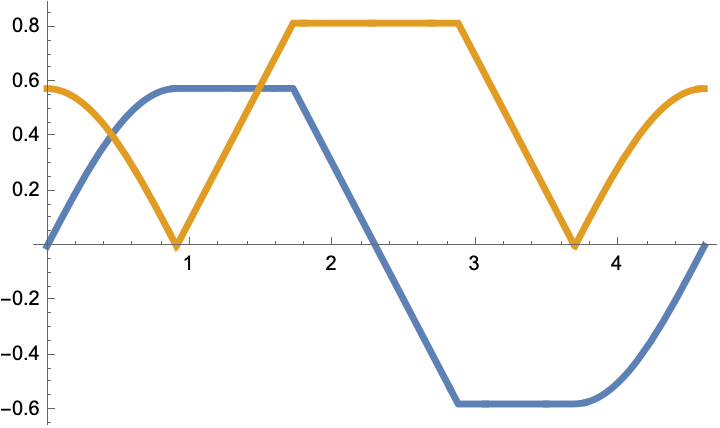}\hskip.2cm \includegraphics[scale=0.42]{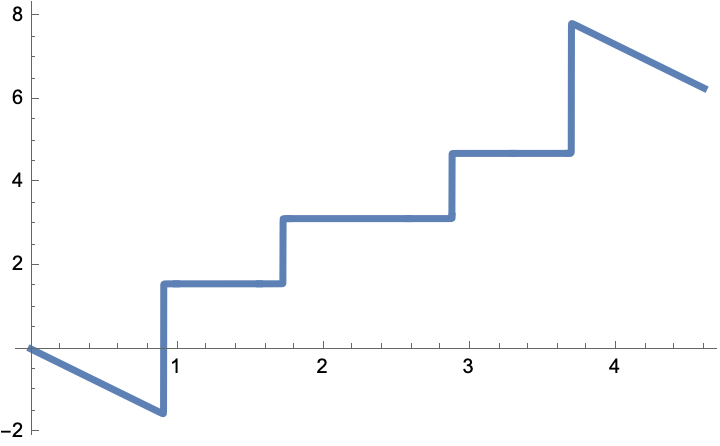}}
\caption{Graph of the profile curve associated with the solution $Z_1=(a,H,T)=(0.577096, -0.707791, 2.30054)$. This solution provides a smooth embedded CMC with curvature $\bfS{4}$ very close to the manifold
$M$ defined in Equation \eqref{mM}.}
 \label{fig:Z1S4}
\end{figure}

Points in the profile curve described in Figure \ref{fig:Z1S4} (image on the left)  that seems to be two vertical segments correspond to points in the smooth CMC hypersurface that are close to the two totally umbilical hypersurfaces in $\bfS{4}$ described as points in the $\bfS{4}$ such that $x_5=\frac{1}{\sqrt{3}}$ and $x_5=-\frac{1}{\sqrt{3}}$. On the other hand, points that seem to be in the horizontal line and in the semi-circle correspond to points in the smooth CMC that are close to Clifford hypersurfaces in $\bfS{4}$. 

 \subsection{A little farther from  $Z_0$} The point
 
 $$Z_2=(a,H,T)=(0.514328, -0.844304, 2.26274)$$ 
 
 solves the system of equations  \eqref{theaHTsytem}. Figure \ref{fig:Z2S4} shows the graphs of the profile curve  and the functions $f_1,f_2$ and $\theta$ associated with the parameters in $Z_2$. We can see that this time we have an embedded hypersurface with CMC  $-0.844304$. We notice that as we move away from $Z_0$, the values for $H$ becomes more negative, we also notice that $T$ decreses.
 
 \begin{figure}[h]
\centerline{\includegraphics[scale=0.42]{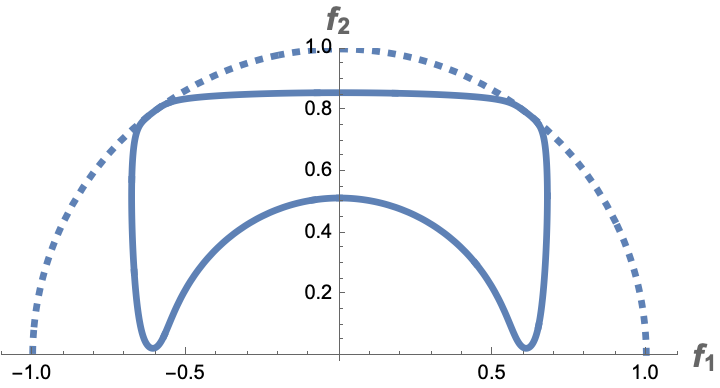} \hskip.2cm \includegraphics[scale=0.42]{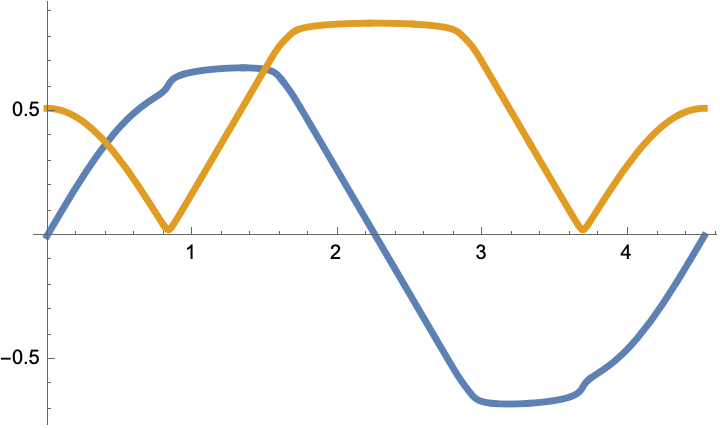}\hskip.2cm \includegraphics[scale=0.42]{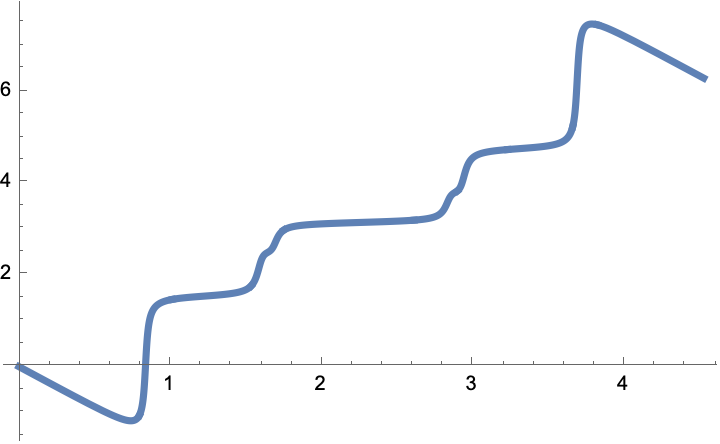}}
\caption{Graph of the profile curve associated with the solution $Z_2=(a,H,T)=(0.514328, -0.844304, 2.26274)$. This solution provides an embedded CMC  hypesurface in$\bfS{4}$} 
\label{fig:Z2S4}
\end{figure}

 \subsection{The lowest value of H in the familiy} The point
 
 $$Z_3=(a,H,T)=(0.433855, -0.947962, 2.22231)$$ 
 
 solves the system of equations  \eqref{theaHTsytem}. Figure \ref{fig:Z3S4} shows the graphs of the profile curve  and the functions $f_1,f_2$ and $\theta$ associated with the parameters in $Z_3$. We can see that this time we have an embedded hypersurface with CMC  $-0.947962$. This value of $H$ will be the lowest in the family. 
  
 \begin{figure}[h]
\centerline{\includegraphics[scale=0.42]{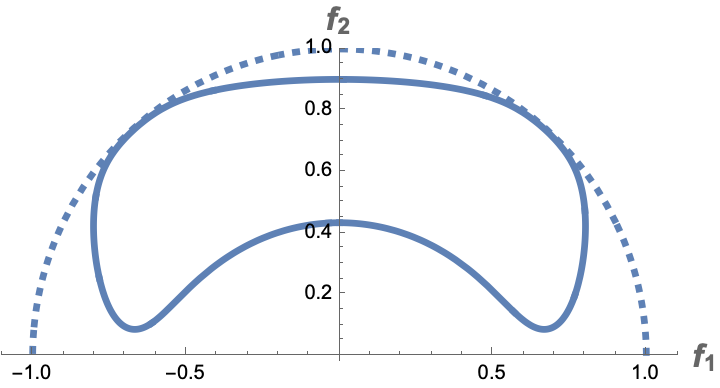} \hskip.2cm \includegraphics[scale=0.42]{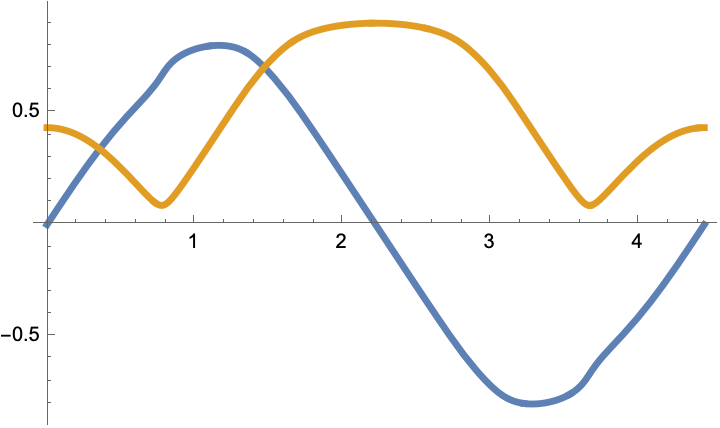}\hskip.2cm \includegraphics[scale=0.42]{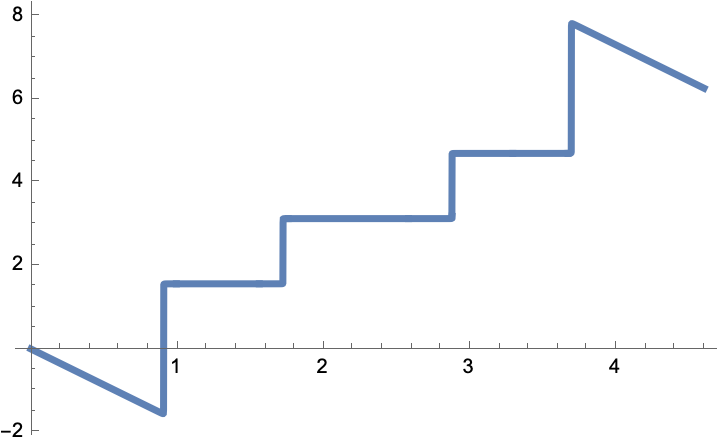}}
\caption{Graph of the profile curve associated with the solution $Z_2=(a,H,T)=(0.433855, -0.947962, 2.22231)$. This solution provides an embedded CMC  hypesurface in $\bfS{4}$} 
\label{fig:Z3S4}
\end{figure}

 \subsection{An still embedded example} The point
 
 $$Z_4=(a,H,T)=(0.635046, -0.258674, 2.37217)$$ 
 
 solves the system of equations  \eqref{theaHTsytem}. Figure \ref{fig:Z4S4} shows the graphs of the profile curve  and the functions $f_1,f_2$ and $\theta$ associated with the parameters in $Z_4$. We can see that this time we have an embedded hypersurface with CMC  $-0.258674$.  $H$ is still negative but it will eventually reach zero and some positive numbers before coming back to positive values close to zero.
 
 \begin{figure}[h]
\centerline{\includegraphics[scale=0.42]{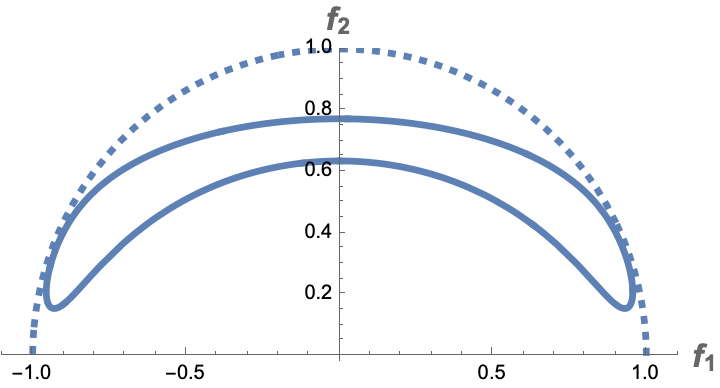} \hskip.2cm \includegraphics[scale=0.42]{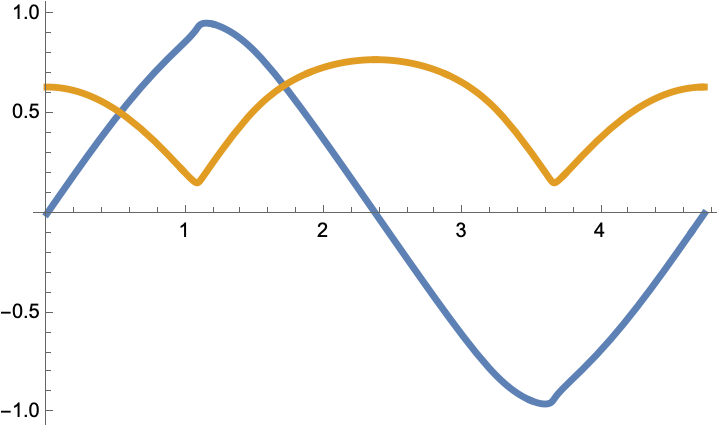}\hskip.2cm \includegraphics[scale=0.42]{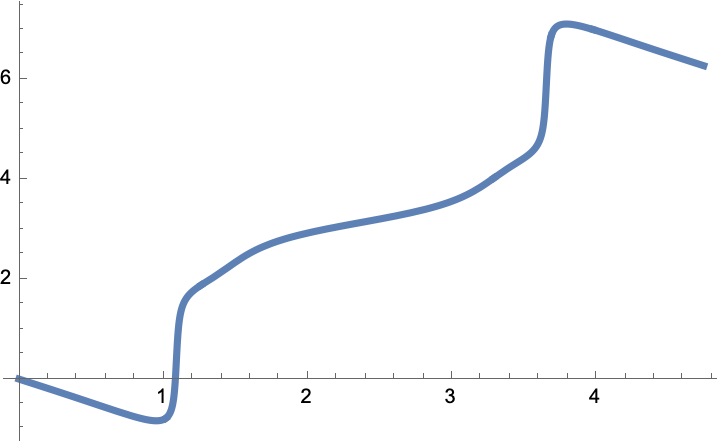}}
\caption{Graph of the profile curve associated with the solution $Z_4=(a,H,T)=(0.635046, -0.258674, 2.37217)$. This solution provides an embedded CMC  hypesurface in $\bfS{4}$} 
\label{fig:Z4S4}
\end{figure}

 \subsection{First non embedded example} The point
 
 $$Z_5=(a,H,T)=(0.707096, -0.0899734, 2.45894)$$ 
 
 solves the system of equations  \eqref{theaHTsytem}. Figure \ref{fig:Z5S4} shows the graphs of the profile curve  and the functions $f_1,f_2$ and $\theta$ associated with the parameters in $Z_5$. We can see that this time we have an embedded hypersurface with CMC  $-0.0899734$. The profile curve is not longer embedded and therefore, the hypersurface is not embedded either. 
 
 \begin{figure}[h]
\centerline{\includegraphics[scale=0.42]{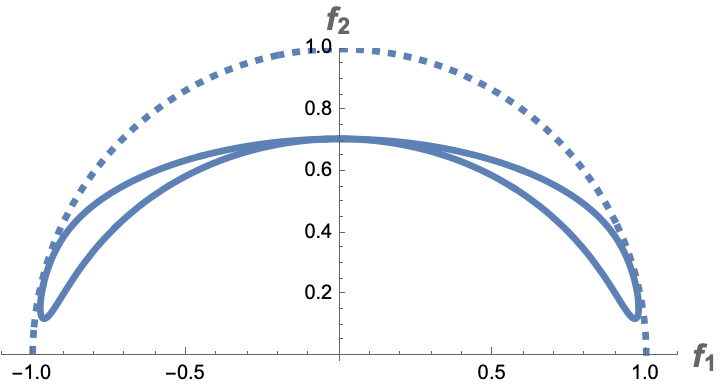} \hskip.2cm \includegraphics[scale=0.42]{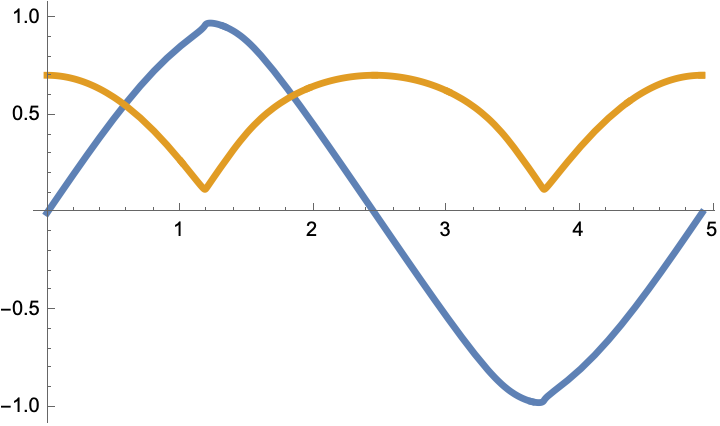}\hskip.2cm \includegraphics[scale=0.42]{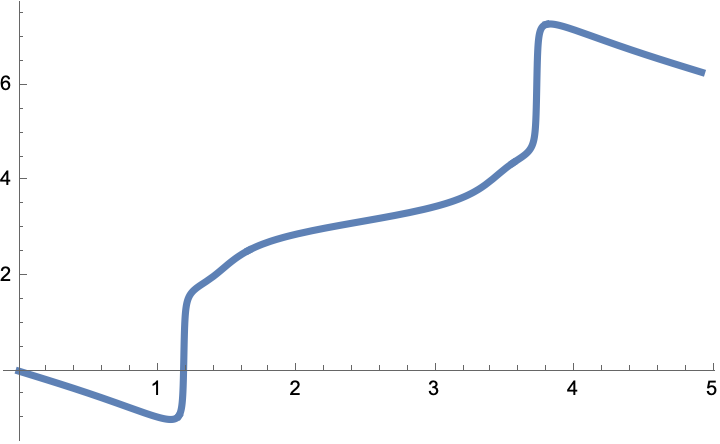}}
\caption{Graph of the profile curve associated with the solution $Z_5=(a,H,T)=(0.707096, -0.0899734, 2.45894)$. This is the first nonembedded CMC hypesurface in $\bfS{4}$ in the family} 
\label{fig:Z5S4}
\end{figure}

 \subsection{The minimal example}  the following  point provides the only minimal example in this family
 
 $$Z_6=(a,H,T)=(0.73801,0.,2.51519)$$ 
 
 solves the system of equations  \eqref{theaHTsytem}. Figure \ref{fig:Z6S4} shows the graphs of the profile curve  and the functions $f_1,f_2$ and $\theta$ associated with the parameters in $Z_6$. We can see that this time we have nonembedded minimal hypersurface.
 
 Had this example been embedded, it would have constituted a counterexample to the Carlotto-Schulz conjecture, which states that the hypertorus is the only embedding of
  $S^1\times S^1\times S^1$ in $\bfS{4}$ \cite{CS}.

 \begin{figure}[h]
\centerline{\includegraphics[scale=0.42]{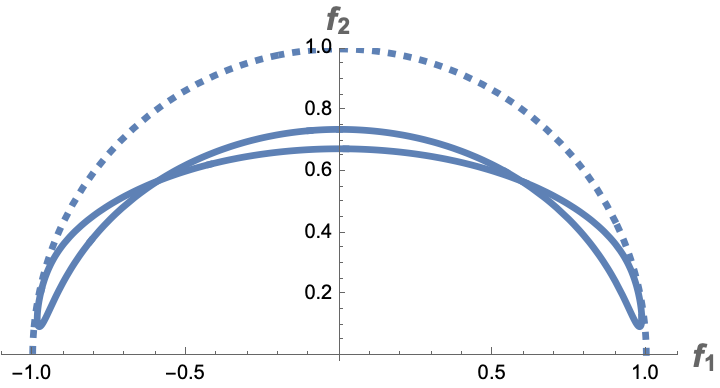} \hskip.2cm \includegraphics[scale=0.42]{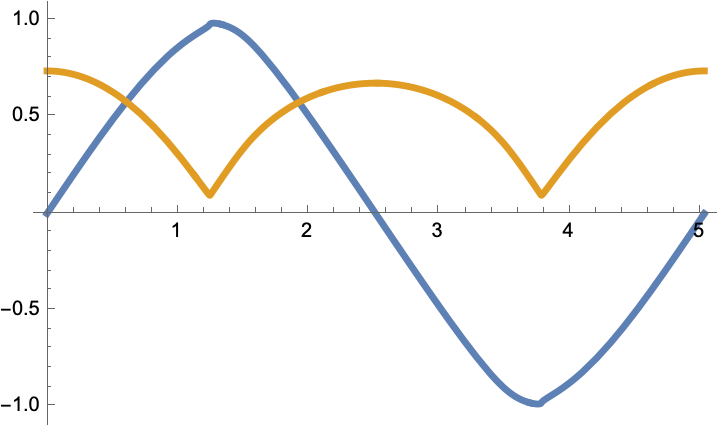}\hskip.2cm \includegraphics[scale=0.42]{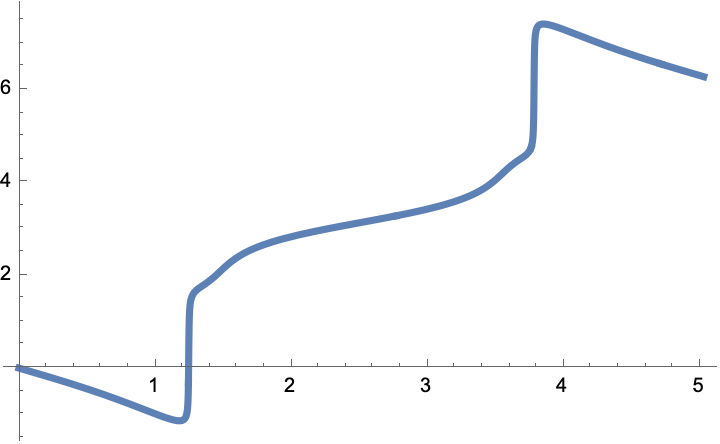}}
\caption{Graph of the profile curve associated with the solution $Z_6=(a,H,T)=(0.73801,0.,2.51519)$. This solution provides an immersed (non embedded) minimal hypersurface in  $\bfS{4}$} 
\label{fig:Z6S4}
\end{figure}

 \subsection{$H$ now is positive and continues increasing} After reaching the value $H=0$, the value of the CMC continue increasing. The piont
 $$Z_7=(a,H,T)=({0.745402, 0.0299556, 2.54038})$$ 
 
 solves the system of equations  \eqref{theaHTsytem}. Figure \ref{fig:Z7S4} shows the graphs of the profile curve  and the functions $f_1,f_2$ and $\theta$ associated with the parameters in $Z_7$.
  
 \begin{figure}[h]
\centerline{\includegraphics[scale=0.42]{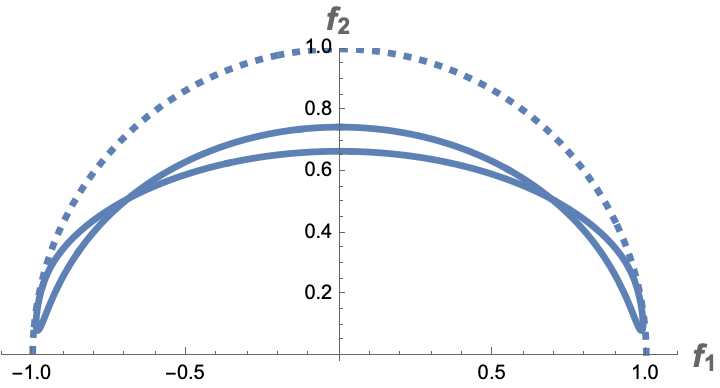} \hskip.2cm \includegraphics[scale=0.42]{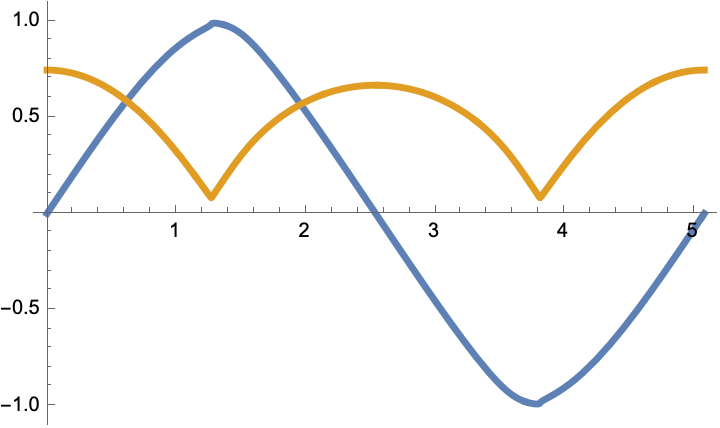}\hskip.2cm \includegraphics[scale=0.42]{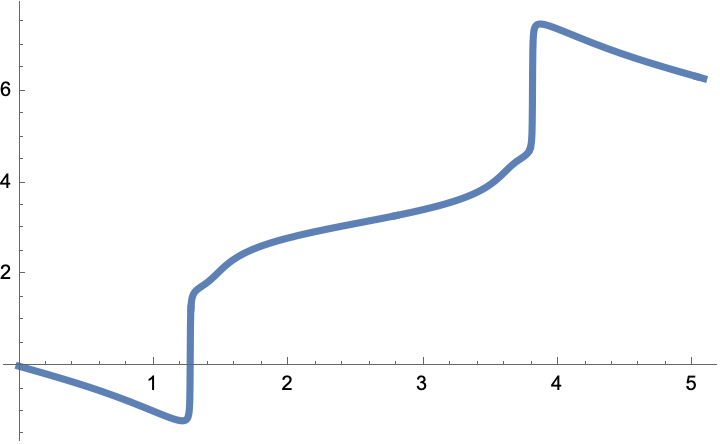}}
\caption{Graph of the profile curve associated with the solution $Z_7=(a,H,T)=({0.745402, 0.0299556, 2.54038})$. This solution provides an embedded CMC with curvature $\bfS{4}$} 
\label{fig:Z7S4}
\end{figure}

 \subsection{The maximum value of H}  The piont
 $$Z_8=(a,H,T)=({0.743855, 0.0565645, 2.5915})$$ 
 
 solves the system of equations  \eqref{theaHTsytem} and provide the example in the family with most positive mean curvature. Figure \ref{fig:Z8S4} shows the graphs of the profile curve  and the functions $f_1,f_2$ and $\theta$ associated with the parameters in $Z_3$. 
  
 \begin{figure}[h]
\centerline{\includegraphics[scale=0.42]{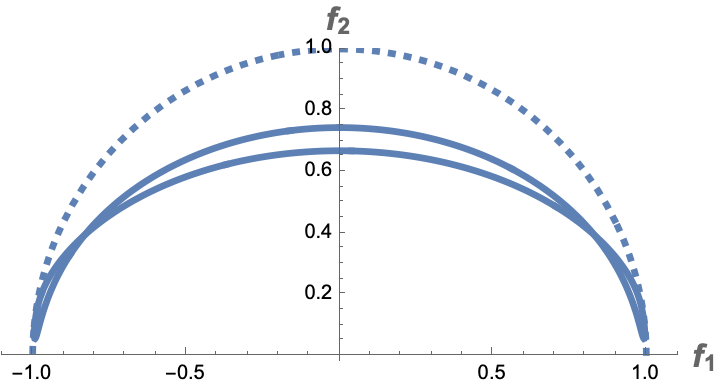} \hskip.2cm \includegraphics[scale=0.42]{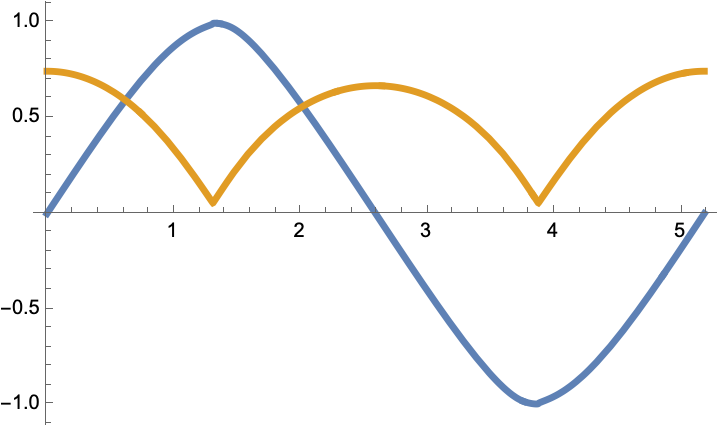}\hskip.2cm \includegraphics[scale=0.42]{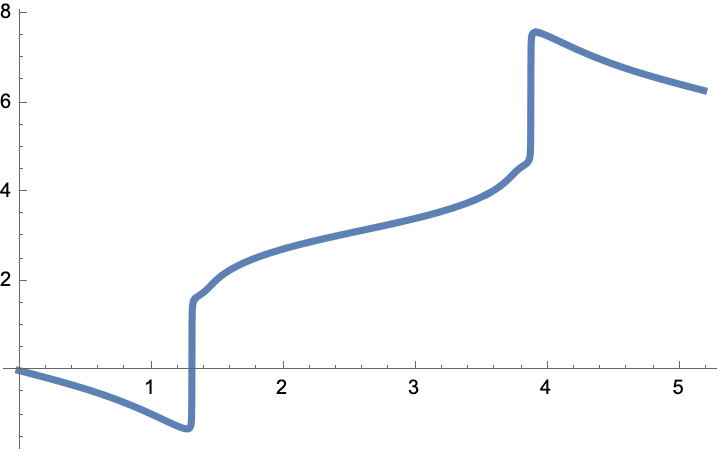}}
\caption{Graph of the profile curve associated with the solution $Z_8=(a,H,T)=({0.743855, 0.0565645, 2.5915})$. This solution provides a CMC with curvature $\bfS{4}$} 
\label{fig:Z8S4}
\end{figure}
 \subsection{Now $H$ starts decreasing to zero} The point
 
 $$Z_9=(a,H,T)=(0.720997, 0.0299491, 2.64565)$$ 
 
 solves the system of equations  \eqref{theaHTsytem}. Figure \ref{fig:Z9S4} shows the graphs of the profile curve  and the functions $f_1,f_2$ and $\theta$ associated with the parameters in $Z_9$. We can see that the profile curve starts looking like half of the ellipse $E$ defined in Remark \ref{rem:singsolellipse}. 
 \begin{figure}[h]
\centerline{\includegraphics[scale=0.42]{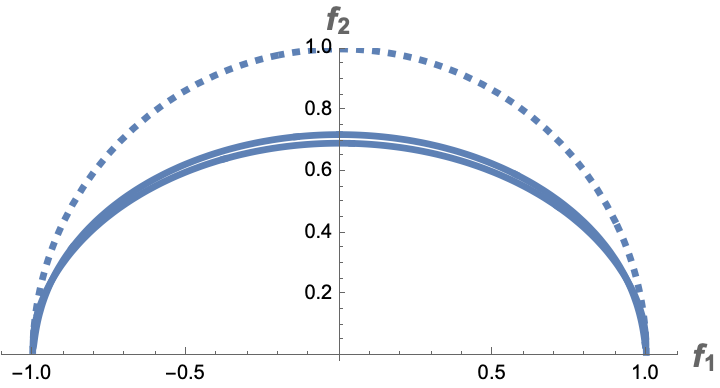} \hskip.2cm \includegraphics[scale=0.42]{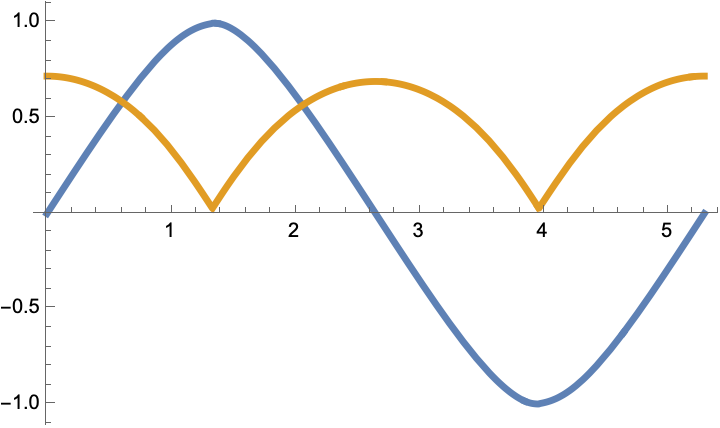}\hskip.2cm \includegraphics[scale=0.42]{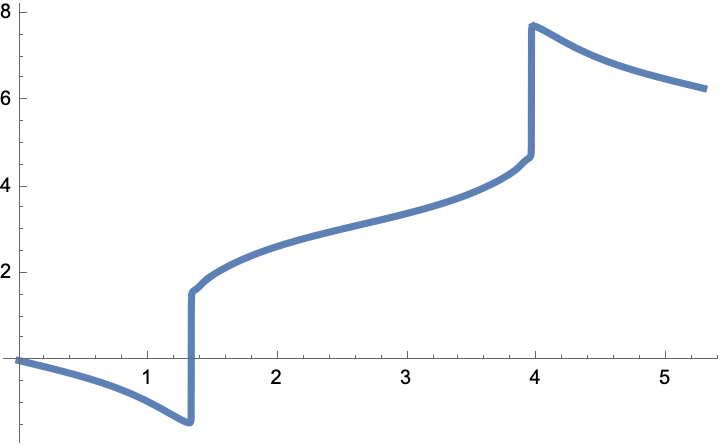}}
\caption{Graph of the profile curve associated with the solution $Z_9=(a,H,T)=(0.720997, 0.0299491, 2.64565)$. This solution provides a CMC with curvature $\bfS{4}$} 
\label{fig:Z9S4}
\end{figure}

 \subsection{Near the end} The point
 
 $$Z_{10}=(a,H,T)=(0.703734, 0.000563715, 2.70305)$$ 
 
 solves the system of equations  \eqref{theaHTsytem}. As we can see, this point is very close to the point $Z_f$ defined in  Equation \eqref{pZf}. Figure \ref{fig:Z10S4} shows the graphs of the profile curve  and the functions $f_1,f_2$ and $\theta$ associated with the parameters in $Z_{10}$.  For the naked eye, if we superpose the graphs of the solutions of $f_1$ and $f_2$ in Figures \ref{fig:gfMf} and \ref{fig:Z10S4}  we could not tell the difference between these to functions, nevertheless  the functions $f_1$ and $f_2$ and $\theta$ in Figures \ref{fig:Z1S4} are $C^\infty$. 

 \begin{figure}[h]
\centerline{\includegraphics[scale=0.42]{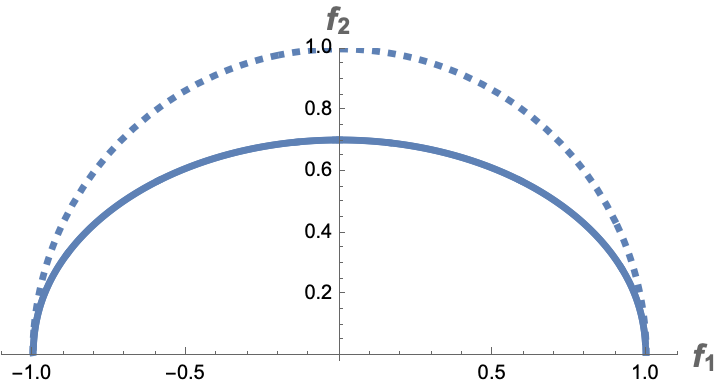} \hskip.2cm \includegraphics[scale=0.42]{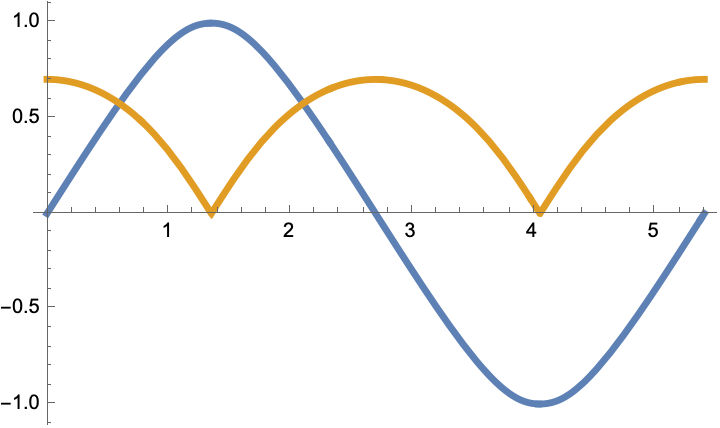}\hskip.2cm \includegraphics[scale=0.42]{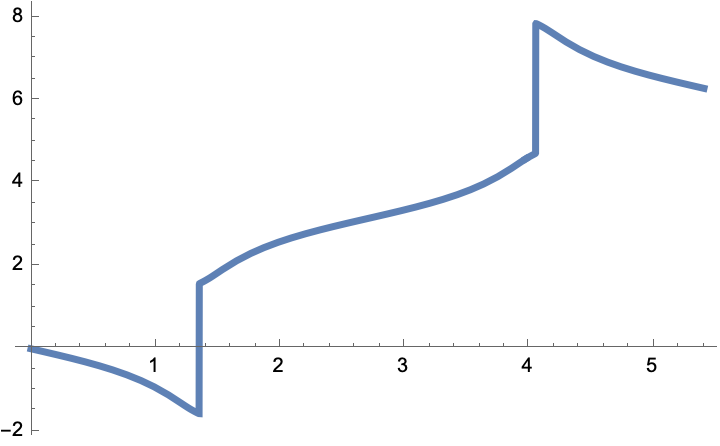}}
\caption{Graph of the profile curve associated with the solution $Z_{10}=(a,H,T)=(0.703734, 0.000563715, 2.70305)$. This hypersurface is very close to minimal hypersurface  $M_f\subset \bfS{4}$ defined in Equation \eqref{mMf}. } 
\label{fig:Z10S4}
\end{figure}

%

\end{document}